\documentclass[12pt,reqno]{amsart}

\usepackage[usenames, dvipsnames]{color}
\usepackage{amsmath, amsthm, amssymb}
\usepackage{amsfonts}
\usepackage[ansinew]{inputenc}
\usepackage{graphicx}
\usepackage[english]{babel}
\usepackage{enumerate}
\usepackage{hyperref}
\theoremstyle{plain}
\newtheorem{thm}{Theorem}[section]

\newtheorem{lem}[thm]{Lemma}
\newtheorem{prop}[thm]{Proposition}
\newtheorem{defn}[thm]{Definition}

\theoremstyle{remark}
\newtheorem{rem}{Remark}

\numberwithin{equation}{section}

\newcommand{\average}{{\mathchoice {\kern1ex\vcenter{\hrule height.4pt
width 6pt depth0pt} \kern-9.7pt} {\kern1ex\vcenter{\hrule
height.4pt width 4.3pt depth0pt} \kern-7pt} {} {} }}

\def\R{\mathbb{R}}

\begin{document}

\title[$C^{2s}$ regularity for fully nonlinear nonlocal equations]{$C^{2s}$ regularity for fully nonlinear nonlocal equations with bounded right hand side}

\author{Hern\'an Vivas}

\address{Instituto de Investigaciones Matem\'{a}ticas Luis A. Santal\'{o} - CONICET, Ciudad Universitaria, Pabell\'{o}n I (1428) Av. Cantilo s/n - Buenos Aires, Argentina}

\address{Centro Marplatense de Investigaciones Matem\'aticas - CIC, De\'an Funes 3350, 7600 Mar del Plata, Argentina}

\email{havivas@mdp.edu.ar}



\begin{abstract}
We establish sharp $C^{2s}$ interior regularity estimates for solutions of fully nonlinear nonlocal equations with bounded right hand side. More precisely, we show that  if $I$ is a fully nonlinear nonlocal concave or convex elliptic operator and $f\in L^\infty(B_1)$ then
\[
Iu=f\quad\textrm{ in }\quad B_1 \quad \Rightarrow\quad u\in C^{2s}(B_{1/2}).
\]
This result generalizes the linear counterpart proved by Ros-Oton and Serra and extends previous available results for fully nonlinear nonlocal operators. As an application, we get a basic regularity estimate for the nonlocal two membranes problem.

\end{abstract}

\maketitle

\section{Introduction}

The aim of this work is to prove $C^{2s}$ interior regularity estimates for solutions of nonlocal fully nonlinear concave (or convex) elliptic equations with bounded right hand side. More precisely, our main result states that if $I$ is a nonlocal fully nonlinear concave (or convex) elliptic operator of order $2s$ (see Definition \ref{def.ellip}) and $f\in L^\infty(B_1)$ then
\[
Iu=f\quad\textrm{ in }\quad B_1\quad \Rightarrow\quad u\in C^{2s}(B_{1/2})
\]
where the equation is to be understood in the viscosity sense, see Definition \ref{def.visc}.

Nonlocal equations have attracted much interest both in the PDE and Probability communities in the past twenty years or so, mainly due to the wide range of problems arising from Physics, Biology, Engineering and Finance (among others) that seem to be well described by these type of models. Indeed, nonlocal equations arise naturally in the context of L\'evy processes, a type of stochastic process in which the path is not necessarily continuous and solutions are allowed to jump; these processes proved to be very useful for describing diffuson-type phenomena which present features that are not captured by the classical Browinan Motion. We refer the reader to \cite{humphries} for examples related to search/pretador-prey models in Biology and to the introductory book of Bucur and Valdinoci \cite{BV16} for a very enlightenig description problems coming from Physics, such as water waves and the Schr\"odinger equation or nonlocal phase transitions, as well as profuse references for other models. Furthermore, Ros-Oton's \cite{Ro1} and Garofalo's \cite{Ga} survey papers provide an overview of the mathematical landmark in the subject and extensive references. Finally, we recomend the nice survey \cite{Ap} and the references therein for an overview of the probabilistic approach to these type of problems and the phenomena related to them and the book \cite{CT16} for a thorough discussion of L\'evy processes with jumps in the context of Mathematical Finance. 

From our (PDE) perspective, the basic nonlocal operators that are of interest are of the following form:
\begin{equation}\label{eq.L}
L u(x)=\int_{\mathbb{R}^n}(u(x+y)+u(x-y)-2u(x))K(y)\:dy
\end{equation}
where some prescriptions are set on the kernel $K$ to ensure ellipticity (see \eqref{eq.ker} below). Note carefully that in order to compute the value of the operator at some point $x$ we need information of the values of the function $u$ in the whole space $\mathbb{R}^n$, hence the term ``nonlocal''.

Although the development of some aspects of the theory of nonlocal (or fractional, or integro-differential) operators from the point of view of Analysis can be traced back several decades, either from a Harmonic Analysis point of view as found in the book by Stein \cite{St} or a potential theoretic approach as in the classical reference of Landkof \cite{Land}, it was the seminal work of Caffarelli and Silvestre in 2009 \cite{CS1} that opened up the way for a (so to speak) modern regularity theory for elliptic nonlocal equations. That paper started a program, continued in \cite{CS2} and \cite{CS3}, to develop a regularity theory for fully nonlinear nonlocal elliptic equations in a parallel fashion to that of the nowadays classic second order theory, cf. \cite{CC}. Since then, the subject has been and continues to be a very active research area, as the reader can see in the surveys mentioned above. 

In particular, in the aforementioned paper of Caffarelli and Silvestre the concept of ellipticity is made precise for integro-differential operators. This is done somehow extending the local definition of ellipticity; indeed, recall that if $F$ is a fully nonlinear \emph{second order} operator, $F$ is said to be (uniformly) elliptic if it satisfies
\[
\mathcal{M}^-(X-Y)\leq F(X)-F(Y)\leq \mathcal{M}^+(X-Y)
\]  
for any pair of symmetric matrices $X$ and $Y$ where $\mathcal{M}^-$ and $\mathcal{M}^+$ are the Pucci extremal operators defined as the infimum and supremum over all linear elliptic operators (or equivalently by the sum of the negative and positive eigenvalues multiplied by the respective ellipticity constants, see \cite{CC}). In the integro-differential setting, given some class of linear nonlocal operators $\mathcal{L}$ we can define the analogous Pucci extremal operators by
\begin{equation}\label{eq.npucci}
\mathcal{M}^+_{\mathcal{L}}u(x):=\sup_{L\in\mathcal{L}}\: L u(x)\quad\textrm{ and }\quad \mathcal{M}^-_{\mathcal{L}}u(x):=\inf_{L\in\mathcal{L}}\: L u(x)
\end{equation}
and give the analogous definition of ellipticity (see Definition \ref{def.ellip} below). The crucial notion of viscosity solution can also be adapted to this context (Definition \ref{def.visc}).

In this paper, we are interested in studying regularity theory for viscosity solutions of concave fully nonlinear integro-differential operators of the form
\begin{equation}\label{eq.L1}
Iu(x)=\inf_{\alpha\in\mathcal{A}} L_\alpha u(x)
\end{equation}
where, for any $\alpha\in\mathcal{A}$, $L_\alpha$ is an operator of the form \eqref{eq.L} with kernel $K_\alpha$ and each kernel in the family $\{K_\alpha(\cdot)\}_{\alpha\in\mathcal{A}}$ satisfies
\begin{equation}\label{eq.ker}
0<\frac{\lambda}{|y|^{n+2s}}\leq K_\alpha(y)\leq \frac{\Lambda}{|y|^{n+2s}},\quad\quad K_\alpha(y)=K_\alpha(-y)
\end{equation}
for some \emph{ellipticity constants} $\lambda$ and $\Lambda$. These restrictions are going to be assumed hereafter and, following \cite{CS1}, such operators are going to be said to belong to the class $\mathcal{L}_0$. We point out that operators in $\mathcal{L}_0$ lie between to multiples of the \emph{Fractional Laplacian}:
\[
-(-\Delta)^su(x):= C_{n,s}\int\frac{u(x+y)+u(x-y)-2u(x)}{|y|^{n+2s}}\:dy,
\]
arguably the most important nonlocal operator from the elliptic PDE viewpoint. On the other hand, the symmetry assumption on the kernels is standard although there are available results for nonsymetric kernels, see for instance \cite{CD}. Also, since whenever $I$ is concave $-I$ is convex we will omit the convex operators from here forward as the proofs of the regularity results are equivalent.

Since the class of linear nonlocal operators is much richer than the local one, we have that  different choices in the class $\mathcal{L}$ of linear operators in the definition \eqref{eq.npucci} give rise to different ``ellipticity classes''. The class $\mathcal{L}_0$ is somehow the largest class for which regularity theory can be developed, since no assumptions are made on the smoothness of the kernel outside the origin and they can be highly oscillatory and irregular. They have been therefore referred to as \emph{rough} kernels (see \cite{K} or \cite{Se1}).  

The main result of this paper is summarized in the following theorem: 

\begin{thm}\label{thm.c2s}
Let $s\in(0,1)$, $I$ a nonlocal operator of the form \eqref{eq.L1}-\eqref{eq.ker} and $u$ be a bounded viscosity solution of 
\[
Iu=f\quad\textrm{ in }\quad B_1.
\] 
Assume that $f\in L^\infty(B_1)$ and that it is continuous in $B_1$ (no modulus of continuity is assumed). 

Then 
\begin{enumerate}
\item if $s\neq 1/2$, $u\in C^{2s}(B_{1/2})$ and 
\begin{equation}\label{eq.2}
\|u\|_{C^{2s}(B_{1/2})}\leq C\left(\|f\|_{L^\infty(B_1)}+\|u\|_{L^\infty(\mathbb{R}^n)}\right)
\end{equation} 
\item if $s= 1/2$, $u\in C^{2s-\varepsilon}(B_{1/2})$ for any $\varepsilon>0$ and 
\begin{equation}\label{eq.3}
\|u\|_{C^{2s-\varepsilon}(B_{1/2})}\leq C\left(\|f\|_{L^\infty(B_1)}+\|u\|_{L^\infty(\mathbb{R}^n)}\right).
\end{equation} 
\end{enumerate}
The constant in \eqref{eq.2} depends only on $n,s,\lambda,\Lambda$ and the constant in \eqref{eq.3} depends also on $\varepsilon$. 
\end{thm}

Let us make some remarks on this theorem. First, it extends known results available for linear operators to the nonlinear setting; indeed our result is exactly the same as the one obtained by Ros-Oton and Serra in the linear case, cf. \cite{RS1}, Theorem 1.1 (b). It is a curiosity worth pointing out that for $s\neq 1/2$ we get regularity of order $2s$, unlike the $2s-\varepsilon$ that we get otherwise or the $C^{1,1-\varepsilon}$ that is achieved in the local case. 

Moreover, our result gives the natural extension of the results by Krivenstov \cite{K} and Serra \cite{Se1} (in the elliptic setting). Indeed, in the first of these papers, $C^{1,\alpha}$ regularity is obtained for operators of order $s>1/2$. We cannot expect this result to be improved substantially as his operators are not assumed to be convex or concave. In Serra's work he proves $C^\beta$ regularity with 
\[
\beta:=\min\{1+\alpha,2s\}-\varepsilon
\]
for any $\varepsilon>0$ with no restricition on the exponent and for parabolic equations, but again with no futher assumptions on the operators, hence the restriction on the regularity. 

In this paper we follow the ideas presented by Serra in the aforecited paper where, to the best to the author's knowledge, they appeared for the first time in the context of nonlocal equations. Serra's proof is based on a compactness argument combined with a Liouville theorem used to classify global solutions. This approach proved to be very useful in the context of nonlocal equations, see for instance the work of Ros-Oton and Serra, \cite{RS1} \cite{RS2}, Fern\'andez Real and Ros-Oton \cite{FR1} or Ros-Oton and the author \cite{RV} for some instances where this technique was used successfully to tackle regularity issues. 

The passage from the $C^{1,\alpha}$ regularity in \cite{Se1} to our $C^{2s}$ regularity is not immediate (of course, assuming that $2s>1+\alpha$), in the sense that the strategy used to prove a Liouville theorem in \cite{Se1} is not good for our higher regularity goal. Here we have to appeal to the higher order regularity results in \cite{Se2} instead to prove a such a result, namely that global solutions with an appropriate growth have to be planes. 

Regarding the hypothesis, they are essentially sharp: the continuity assumption on the right hand side is standard for the viscosity solution setting and \emph{it will be assumed hereafter}. However, we do not impose nor use any sort of modulus of continuity on $f$ and our estimates depend only on its $L^\infty$ norm. 

We finish this introduction with some definitions and preliminary remarks that will be used throughout the paper. As mentioned before, the definition of ellipticity follows naturally once we have the definition of the Pucci extremal operators: 
\begin{defn}\cite{CS1}\label{def.ellip}
Let $\mathcal{L}$ a class of linear integro-differential operators such that any $K\in\mathcal{L}$ satisfies
\[
\int_{\mathbb{R}^n}\frac{|y|^2}{1+|y|^2}K(y)\:dy\leq C<\infty
\]
for some constant $C>0$. 

An elliptic operator $I$ with respect to $\mathcal{L}$ is an operator with the following properties:
\begin{enumerate}

\item if $u$ is a bounded function and $C^{1,1}$ at $x$ then $Iu(x)$ is well defined, 

\item if $u$ is $C^2$ in some open set $\Omega$ (and bounded outside) then $Iu(\cdot)$ is continuous in $\Omega$, 

\item if $u,v$ are to bounded functions, $C^{1,1}$ at $x$, then 
\begin{equation}\label{eq.ellip}
\mathcal{M}^-_{\mathcal{L}}(u-v)(x)\leq Iu(x)-Iv(x)\leq\mathcal{M}^+_\mathcal{L}(u-v)(x).
\end{equation}
\end{enumerate}
\end{defn}  

For these operators, there is a natural notion of viscosity solution given in the definition presented below. Viscosity solutions enjoy, among other properties, good stability behavior in the sense that (appropriate) limits of viscosity solutions are again viscosity solutions, see Lemma \ref{lem.stab}. 
 
\begin{defn}\cite{CS1}\label{def.visc}
Let $I$ be a nonlocal operator, elliptic in the sense Definition \ref{def.ellip} and $f$ a continuous function in $B_1$. A continuous function $u:\mathbb{R}^n\rightarrow\mathbb{R}$ is said to be a viscosity \emph{supersolution} (resp. \emph{subsolution}) of
\[
Iu=f
\]
in $B_1$, also written $Iu\leq f$ (resp. $Iu\geq f$) if the following happens: for any $x_0\in B_1$ and any function $\varphi$ satisfying:   

\begin{itemize}

\item $\varphi(x_0)=u(x_0)$

\item $\varphi(x)<u(x)$ (resp. $\varphi(x)>u(x)$) in some neighborhood $N$ of $x_0$ inside $B_1$

\item $\varphi\in C^2(\overline{N})$

\item $\varphi=u$ in $\mathbb{R}^n\setminus N$
\end{itemize}
we must have
\[
I\varphi(x_0)\leq f(x_0)\quad (\textrm{resp. }I\varphi(x_0)\geq f(x_0)).
\]
\end{defn}

Actually, test functions are not confined to coincide with $u$ outside the neighborhood $N$ but to remain below (or above) it and have a growth that allows the operators to be evaluated at infinity, see the discussion below. 

An important property when proving regularity results is the scaling of the operator. In that direction, it is easy to check that the scaling properties of nonlocal operators continue to hold for viscosity solutions, so that if $u$ satisfies for instance
\[
\mathcal{M}_{\mathcal{L}}^+u(x)\geq f(x)
\]
in the viscosity sense, then $u_r(x):=u(rx)$ satisfies
\[
\mathcal{M}_{\mathcal{L}}^+u_r(x)\geq r^{2s}f(rx)
\]
in the viscosity sense.

We simplify the notation hereafter by setting
\[
\mathcal{M}^+:=\mathcal{M}^+_{\mathcal{L}_0}\quad\textrm{ and }\quad\mathcal{M}^-:=\mathcal{M}^-_{\mathcal{L}_0}
\]
and point out that for this class of kernels the extremal operators have explicit formulae. In fact, it is easy to see that
\begin{equation}\label{eq.form+}
\mathcal{M}^+u(x)=\int\frac{\Lambda\delta(u,x,y)^+-\lambda\delta(u,x,y)^-}{|y|^{n+2s}}\:dy
\end{equation}
and
\begin{equation}\label{eq.form-}
\mathcal{M}^-u(x)=\int\frac{\lambda\delta(u,x,y)^+-\Lambda\delta(u,x,y)^-}{|y|^{n+2s}}\:dy
\end{equation}
where
\[
\delta(u,x,y):=u(x+y)+u(x-y)-2u(x)
\]
and for a real number $a\in\R^n, a^+:=\max\{a,0\},a^-:=-\min\{a,0\}$.

We would also like to point out that operators of the form \eqref{eq.L} actually allow some growth on the functions to be evaluated. Indeed, it is enough to ask for the function $u$ to be integrable against a polynomial that grows like $n+2s$, so that we can require for instance
\[
\int\frac{|u(y)|}{1+|y|^{n+2s}}\:dy<\infty
\]
instead of boundedness in Definition \ref{def.ellip}. Such functions are said to belong to $L^1(\omega_s)$, where 
\[
\omega_s(x):=\frac{1}{1+|x|^{n+2s}}.
\] 
In the sequel, we will use the $L^\infty$ norm and state our results for bounded functions for simplicity, but all the proofs can be easily recast to fit into this more general setting.  

Finally, a word about notation: for $\gamma>0$, we will use the brackets $[\cdot]_{C^\gamma}$ to denote the usual H\"older seminorms. Therefore, if $0<\gamma<1$ and $\Omega$ is an open set
\[
[u]_{C^\gamma(\Omega)}:=\sup_{\substack{x\neq y \\ x,y\in\Omega}}\frac{|u(y)-u(x)|}{|y-x|^\gamma}
\]
and the full norm is defined by 
\[
\|u\|_{C^\gamma(\Omega)}:=\|u\|_{L^\infty(\Omega)}+[u]_{C^\gamma(\Omega)}.
\]

As we will have to deal with cases for which $\gamma$ will be bigger than one (and always non-integer), it simplifies notation to denote by $[u]_{C^\gamma}$ the $C^{\gamma-\lfloor\gamma\rfloor}$ of the derivatives of order $\lfloor\gamma\rfloor$ of $u$ (with $\lfloor a\rfloor$ denoting the integer part of $a$). Hence, if for instance $\gamma\in(1,2)$ then $[u]_{C^\gamma}$ denotes the H\"older  seminorm $[\nabla u]_{C^{\gamma-1}}$ and the full norm is defined by
\[
\|u\|_{C^\gamma(\Omega)}:=\|u\|_{L^\infty(\Omega)}+\|\nabla u\|_{L^\infty(\Omega)}+[\nabla u]_{C^{\gamma-1}(\Omega)}.
\]

The rest of the paper is organized as follows: Section \ref{sec.liou} is concerned with the proof of the Liouville theorem that allows us to classify global solutions; in Section \ref{sec.main}, after some necessary preliminary results, we give the proof of Theorem \ref{thm.c2s}; finally, in Section \ref{sec.two} we give an application to our result that provides a basic regularity estimate for the two membranes problem for fully nonlinear nonlocal operators. 

\section{A Liouville theorem}\label{sec.liou}
  
Liouville type theorems are classification results that express the ``rigidity'' of elliptic equations, in the sense that global solutions with some prescribed growth at infinity have to be polynomials. In this section we prove one such result that will be one of the main ingredients in the proof of Theorem \ref{thm.c2s}. 
 
\begin{thm}\label{thm.liou}
Let $s\in(0,1)$ and let $u$ be a $C^{2,\varepsilon}$ function for some $\varepsilon>0$ satisfying the following:
\begin{enumerate}
\item for any $h\in\mathbb{R}^n$ 
\begin{equation}\label{eq.liou2}
\mathcal{M}^-(u(\cdot+h)-u(\cdot))\leq 0\leq \mathcal{M}^+(u(\cdot+h)-u(\cdot))\quad\textrm{ in }\quad \mathbb{R}^n,
\end{equation}

\item for any $\mu\in L^1(\mathbb{R}^n)$ with compact support and $\int_{\mathbb{R}^n}\mu(h)\: dh=1$ 
\begin{equation}\label{eq.liou3}
\mathcal{M}^+\Big(\int u(\cdot+h)\mu(h)\:dh-u\Big)\geq 0\quad\textrm{ in }\quad \mathbb{R}^n. 
\end{equation}

\end{enumerate}
Assume further that there exists a constant $C>0$ and some $\alpha<2s$ such that
\[
[u]_{C^\alpha(B_R)}\leq CR^\gamma
\]
for all $R\geq 1$ with $\gamma:=2s-\alpha$.

Then 
\[
u(x)=p\cdot x+q
\]
for some $p\in\mathbb{R}^n$ and $q\in\mathbb{R}$ with $p=0$ if $s<1/2$.
\end{thm}

\begin{rem}
A comment about hypothesis (2) is in order: the reason why concavity opens up the way for a $C^{2,\alpha}$ estimate in the local (second order) case is that when $F$ is concave the second derivatives of a solution are themselves subsolutions of an equation with bounded measurable coefficients. In our nonlocal setting, the situation is similar in the sense that hypothesis (1)-(2) comprise the relevant information about being a solution to a concave fully nonlinear equation. This was not needed for the lower $C^{1,\alpha}$ regularity estimates of previous works but it is in ours if we assume $2s>1+\alpha$ (see \cite{CS3}, \cite{Se2}).  
\end{rem}

\begin{proof}[Proof of Theorem \ref{thm.liou}.]
As is customary in Liouville type theorems, we want to apply interior estimates to rescalings of $u$ in larger and larger balls. Therefore we define
\[
u_r(x):=\frac{1}{r^{2s}}u(rx)
\]
and the forthcoming estimates are to be considered \emph{a priori} estimates.

Now, $u_r$ satisfies
\[
\mathcal{M}^-(u_r(x+h)-u_r(x))= \mathcal{M}^-(u(r(x+h))-u(rx))\leq 0
\]
owing to \eqref{eq.liou2} and the natural scaling of the nonlocal operators and analogously for $\mathcal{M}^+$ so that  
\[
\mathcal{M}^-(u_r(\cdot+h)-u_r(\cdot))\leq 0\leq \mathcal{M}^+(u_r(\cdot+h)-u_r(\cdot))\quad\textrm{ in }\quad \mathbb{R}^n.
\]
Similarly, it follows that $u_r$ satisfies \eqref{eq.liou3} and 
\[
[u_r]_{C^{\alpha}(B_R)}=\frac{1}{r^{2s}}[u(r\cdot)]_{C^{\alpha}(B_R)}=\frac{1}{r^{2s-\alpha}}[u]_{C^{\alpha}(B_{rR})}\leq CR^\gamma
\]
so the growth also follows. 

Therefore, by virtue of the regularity results in \cite{Se2} (in particular we can repeat the argument in the proof of Theorem 2.1) we have that 
\[
[u_r]_{C^{2s+\tilde{\alpha}}(B_{1/4})}\leq C
\]
for some $\tilde{\alpha}>0$ and some universal constant $C>0$. Scaling back this implies s
\[
r^{\tilde{\alpha}}[u]_{C^{2s+\tilde{\alpha}}(B_{r/4})}\leq C
\]
or
\[
[u]_{C^{2s+\tilde{\alpha}}(B_{r/4})}\leq Cr^{-\tilde{\alpha}}
\]
and if we let $r$ go to $\infty$ we obtain
\[
[u]_{C^{2s+\tilde{\alpha}}(\mathbb{R}^n)}=0.
\]
If $2s+\tilde{\alpha}<2$ this readily implies the desired result. If not we can now use the growth control once more to conclude the proof. 
\end{proof}

\begin{rem}
It is reasonable that we can use Theorem 2.1 in \cite{Se2} for our purposes. Indeed, their solutions ``grow more'' than ours since their aim is to prove higher regularity. The main difference (and constraint for regularity in our case) is the presence of a right hand side, but this is essentially forced to vanish in the limit in our blow up argument (see the proof of Proposition \ref{prop.blow1}).
\end{rem}

\section{Proof of Theorem \ref{thm.c2s}}\label{sec.main}

In this section we present the proof of Theorem \ref{thm.c2s}. Its main ingredients are going to be Propositions \ref{prop.blow1} (for part a)) and  \ref{prop.blow2} (for part b)). Their proofs, in turn, consist on performing a sequence of blow ups with the correct normalization in order to obtain a solution to a homogeneous equation in the whole space whose solutions can be classified via the Liouville theorem \ref{thm.liou} and arrive from there to a contradiction. We start with the following technical lemma that allows us to take limits:

\begin{lem}\label{lem.stab}
Let $\{w_m\}_m$ be a sequence of bounded continuous functions in $\overline B_R$ for some $R>0$. Assume that 

\begin{enumerate}

\item there exist constants $C_m$ such that
\[
\mathcal{M}^+w_m\geq -C_m\quad\textrm{ and }\quad\mathcal{M}^-w_m \leq C_m
\]
in the viscosity sense in $B_R$,

\item there exists a function $w\in C(\mathbb{R}^n)\cap L^1(\omega_s)$ such that
\[
w_m\longrightarrow w
\]
uniformly in $\overline B_R$ and 
\[
\int_{\mathbb{R}^n}\frac{|w_m(x)-w(x)|}{1+|x|^{n+2s}}\:dx\longrightarrow 0
\]
and
\[
C_m\longrightarrow 0
\] 
as $m\rightarrow\infty$.
\end{enumerate}

Then $w$ is a viscosity solution of
\begin{equation}\label{eq.lim}
\mathcal{M}^+w\geq 0 \quad\textrm{ and }\quad\mathcal{M}^-w \leq 0
\end{equation}
in $B_R$.
\end{lem}

\begin{proof}
We will prove the first inequality in \eqref{eq.lim}, the other one is completely analogous. Recall Definition \ref{def.visc} of viscosity solution and consider $x\in B_R$ and a function $\varphi$ that is $C^2$ in some neighborhood $N$ of $x$ (contained in $B_R$). Moreover, suppose $\varphi$ touches $w$ by above at $x$ in $N$, i.e.
\[
\varphi(x)=w(x)\quad\textrm{ and }\quad \varphi(y)>w(y)\quad\textrm{ in }\quad N\setminus\{x\}
\] 
and let 
\[
v(x):=
\left\{
\begin{array}{ccc}
\varphi(x) & \textrm{ if } & x\in N \\
w(x) & \textrm{ if } & x\in \mathbb{R}^n\setminus N 
\end{array}
\right.
\]
Because of the uniform convergence of $\{w_m\}_m$ there exist a sequence of points $\{x_m\}_m\subset B_R$ and a sequence of numbers $\{a_m\}_m$ such that
\begin{enumerate}

\item $x_m\longrightarrow x$ as $m\rightarrow\infty$

\item $a_m\longrightarrow 0$ as $m\rightarrow\infty$

\item $\varphi+a_m$ touches $w_m$ by a above at $x_m$ in $N$ for (all) sufficiently large $m$.

\end{enumerate} 

Therefore, if we define
\[
v_m(x):=
\left\{
\begin{array}{ccc}
\varphi(x)+a_m & \textrm{ if } & x\in N \\
w_m(x) & \textrm{ if } & x\in \mathbb{R}^n\setminus N 
\end{array}
\right.
\]
we have, by hypothesis, 
\begin{equation}\label{eq.lim2}
\mathcal{M}^+v_m(x_m)\geq -C_m.
\end{equation}

Now we use the explicit formula for $\mathcal{M}^+$ given in \eqref{eq.form+}:
\begin{align*}
\mathcal{M}^+v_m(x_m) & = \int_N\frac{\Lambda\delta(\varphi,x_m,y)^+-\lambda\delta(\varphi,x_m,y)^-}{|y|^{n+2s}}\:dy \\
					  & + \int_{\mathbb{R}^n\setminus N}\frac{\Lambda\delta(u_m,x_m,y)^+-\lambda\delta(u_m,x_m,y)^-}{|y|^{n+2s}}\:dy. \\
\end{align*}

Now 
\[
\int_N\frac{\Lambda\delta(\varphi,x_m,y)^+-\lambda\delta(\varphi,x_m,y)^-}{|y|^{n+2s}}\:dy \longrightarrow\int_N\frac{\Lambda\delta(\varphi,x,y)^+-\lambda\delta(\varphi,x,y)^-}{|y|^{n+2s}}\:dy
\]
as $m\rightarrow\infty$ by continuity (recall $\varphi\in C^2$) whereas $\int_{\mathbb{R}^n\setminus N}\frac{\Lambda\delta(u_m,x_m,y)^+-\lambda\delta(u_m,x_m,y)^-}{|y|^{n+2s}}\:dy$ converges to $\int_{\mathbb{R}^n\setminus N}\frac{\Lambda\delta(u,x_m,y)^+-\lambda\delta(u,x_m,y)^-}{|y|^{n+2s}}\:dy$ as $m\rightarrow\infty$ by hypothesis b) and since that the right hand side of \eqref{eq.lim2} converges to 0 we get
\[
\mathcal{M}^+v(x)\geq 0
\]
which is what we wanted to prove.
\end{proof}

The next two Lemmas are concerned with a similar stability property, but they deal with a sequence of different operators, therefore we need to define a notion of convergence of operators which is the following (see \cite{CS2}):
\begin{defn}
Let $\{I_m\}_m$ be a sequence of translation invariant operators of order $2s$ belonging to some fixed ellipticity class and $I$ be an operator in such a class. We say that $I_m$ converges weakly to $I$ if for any $\varepsilon>0$ and for any test function $w$ which is a quadratic polynomial in $B_\varepsilon$ and belongs to $L^1(\omega_s)$ we have
\[
I_mw\longrightarrow Iw\quad\textrm{ uniformly in }\quad \overline{B_{\varepsilon/2}}.
\]
\end{defn} 

Actually, the ellipticity class needs not to be fixed, but for our purposes that case is enough. The same remark holds for the following Lemma, that shows that this type of convergence has the advantage of enjoying a good compactness property. This is Theorem 42 in \cite{CS2}:

\begin{lem}\label{lem.stab2}
If $\{I_m\}_m$ is any sequence of translation invariant operators of order $2s$ belonging to the same ellipticity class and satisfying $I_m0=0$ for all $m$. Then there exists an elliptic operator $I$ belonging to the same ellipticity class such that, up to a subsequence, $I_m$ converges weakly to $I$. 
\end{lem}

Finally, with the aid of the previous two lemmas it is immediate to obtain the following:

\begin{lem}\label{lem.stab3}
Let $\{w_m\}_m$ be a sequence of bounded continuous functions in $\overline B_R$ for some $R>0$. Assume that 

\begin{enumerate}

\item 
\[
I_mw_m=f_m\quad\textrm{ in } B_R
\]
in the viscosity sense in $B_R$ with $f_m$ a continuous and bounded function in $B_R$,

\item there exists a function $w\in C(\mathbb{R}^n)\cap L^1(\omega_s)$ such that
\[
w_m\longrightarrow w
\]
uniformly in $\overline B_R$ and 
\[
\int_{\mathbb{R}^n}\frac{|w_m(x)-w(x)|}{1+|x|^{n+2s}}\:dx\longrightarrow 0,
\]

\item there exists and operator $I$ such that $I_m$ converges weakly to $I$ 

\item 
\[
f_m\longrightarrow0\quad\textrm{ uniformly in } B_R.
\] 
\end{enumerate}

Then $w$ is a viscosity solution of
\[
Iw=0
\]
in $B_R$.
\end{lem}

\begin{proof}
The proof is the same as Lemma 5 in \cite{CS2} and consists on combining an argument like the one in the proof of Lemma \ref{lem.stab} with the weak convergence to control the behavior at infinity. 
\end{proof}

We are now in position to take the main step towards the proof of Theorem \ref{thm.c2s}, which is contained in Propositions \ref{prop.blow1} and \ref{prop.blow2} depending on whether $s\neq1/2$ or $s=1/2$. We start with the former case:

\begin{prop}\label{prop.blow1}
Let $s\in(0,1)$, $s\neq1/2$ and let $I$ be nonlocal operator of the form \eqref{eq.L1}-\eqref{eq.ker}. Let $u$ be a bounded viscosity solution of
\[
Iu=f\quad\textrm{ in }\quad B_1.
\] 
Assume $f\in L^\infty(B_1)$ and $u\in C^\alpha(\mathbb{R}^n)$ with $\lfloor 2s\rfloor<\alpha<2s$. 

Then $u\in C^{2s}(B_{1/2})$ and
\begin{equation}\label{eq.c2s}
[u]_{C^{2s}(B_{1/2})}\leq C(\|f\|_{L^\infty(B_1)}+[u]_{C^\alpha(\mathbb{R}^n)})
\end{equation}
where $C$ is a constant depending only on $s,n,\lambda,\Lambda$ and $\alpha$. 
\end{prop}

The analogous result for $s=1/2$ is the following:

\begin{prop}\label{prop.blow2}
Let $s=1/2$ and $I$ be a convex (or concave) translation invariant nonlocal operator elliptic in the sense of \eqref{eq.ellip}-\eqref{eq.ker}. Let $u$ be a bounded viscosity solution 
\[
Iu=f\quad\textrm{ in }\quad B_1.
\] 

Assume $f\in L^\infty(B_1)$ and $u\in C^\alpha(\mathbb{R}^n)$ with $\lfloor 2s\rfloor<\alpha<2s$. 

Then for any $\varepsilon>0$, $u\in C^{2s-\varepsilon}(B_{1/2})$ and
\begin{equation}\label{eq.c2sep}
[u]_{C^{2s-\varepsilon}(B_{1/2})}\leq C(\|f\|_{L^\infty(B_1)}+[u]_{C^\alpha(\mathbb{R}^n)})
\end{equation}
where $C$ is a constant depending only on $s,n,\lambda,\Lambda,\alpha$ and $\varepsilon$. 
\end{prop}

\begin{proof}[Proof of Proposition \ref{prop.blow1}]

We may assume $\|f\|_{L^\infty(B_1)}+[u]_{C^\alpha(\mathbb{R}^n)}\leq 1$. We argue by contradiction and assume that \eqref{eq.c2s} does not hold. Then there exist sequences $u_k,f_k$ and $I_k$ such that 

\begin{itemize}

\item $I_ku=\inf_{\alpha\in\mathcal{A}_k}L_\alpha u(x)$ with $L_\alpha$ of the form \eqref{eq.L} and $K_\alpha$ satisfying \eqref{eq.ker}, 

\item $u_k$ is a bounded viscosity solution of $I_ku_k=f_k$ in $B_1$,

\item and $\|f_k\|_{L^\infty(B_1)}+[u_k]_{C^\alpha(\mathbb{R}^n)}\leq 1$

\end{itemize}
but 
\begin{equation}\label{eq.blow}
[u_k]_{C^{2s}(B_{1/2})}\longrightarrow\infty\quad\textrm{ as }\quad k\rightarrow\infty.
\end{equation}

Let us set 
\[
\gamma:=2s-\alpha\quad\textrm{ and }\quad \nu:=\lfloor 2s\rfloor
\]
and notice that
\[
\sup_k\sup_{z\in B_{1/2}}\sup_{r>0}r^{-\gamma}[u_k]_{C^\alpha(B_r(z))}=\infty.
\]
Next, we define
\[
\theta(r):= \sup_k\sup_{z\in B_{1/2}}\sup_{r'>r}(r')^{-\gamma}[u_k]_{C^\alpha(B_{r'}(z))}.
\]
Note that 
\begin{enumerate}

\item $\theta$ is nonincreasing in $r$ and is finite for any $r>0$ by the assumptions on $u_k$,

\item $\theta(r)\longrightarrow\infty$ as $r\rightarrow0^+$.

\end{enumerate}  

Hence, we can find a sequence of radii $r'_m\geq1/m$, $k_m$ and $z_m\in B_{1/2}$ such that for each $m\in\mathbb{N}$ 
\begin{equation}\label{eq.nondegm}
(r'_m)^{-\gamma}[u_{k_m}]_{C^\alpha(B_{r'_m}(z_m))}\geq \frac{1}{2}\theta(1/m)\geq\frac{1}{2}\theta(r'_m).
\end{equation}

Next we define $p_{k,r,z}(\cdot-z)$ the polynomial of degree at most $\nu$ that best fits $u_k$ in $B_r(z)$ in the least squares sense, that is $p_{k,r,z}$ minimizes
\[
\int_{B_r(z)}(u_k(x)-p(x))^2\:dx
\]
over all polynomial $p$ of degree at most $\nu$. Note that in particular $p_{k,r,z}$ is either a constant or a linear function.

Now we set
\[
u_m:=u_{k_m}\quad\textrm{ and }\quad p_m:=p_{k_m,r'_m,z_m}
\]
to simplify notation and define the blow-up sequence
\[
v_m(x):=\frac{u_m(z_m+r'_mx)-p_m(r'_m x)}{(r'_m)^{\alpha+\gamma}\theta(r'_m)}.
\] 

Then, the minimization condition implies
\begin{equation}\label{eq.ortm}
\int_{B_1}v_m(x)p(x)\:dx=0
\end{equation}
for any polynomial $p$ of degree at most $\nu$. Also  \eqref{eq.nondegm} implies the following non-degeneracy condition:
\begin{equation}\label{eq.nondeg}
[v_m]_{C^\alpha(B_1)}\geq 1/2
\end{equation}

Moreover, we have the following growth control: 
\begin{equation}\label{eq.growthm}
[v_m]_{C^\alpha(B_R)}\leq CR^\gamma.
\end{equation}
Indeed 
\begin{align*}
[v_m]_{C^\alpha(B_R)} & = \frac{1}{(r'_m)^\gamma\theta(r'_m)} [u_m]_{C^\alpha(B_{r'_mR}(z_m))} \\
 					  & = \frac{R^\gamma}{(r'_mR)^\gamma\theta(r'_m)} [u_m]_{C^\alpha(B_{r'_mR}(z_m))}  \\ 					   					  
 					  & \leq CR^\gamma
\end{align*}
where we used the definition of $\theta$ and its monotonicity. This, when $R=1$ implies $\|v_m-q\|_{L^\infty(B_1)}\leq C$ and therefore we have
\begin{equation}\label{eq.growthm1}
\|v_m\|_{L^\infty(B_1)}\leq C.
\end{equation}

Next we claim that the sequence $\{v_m\}_m$ converges (up to a subsequence) to a function $v$ for which the Liouville theorem \ref{thm.liou} applies. To see this, we first note that $\{v_m\}_m$ is locally uniformly bounded and equicontinuous by \eqref{eq.growthm} and \eqref{eq.growthm1}. Then, the Arzlel\`a-Ascoli theorem assures the existence of a continuous function $v$ such that a subsequence of $\{v_m\}_m$ (which we keep denoting $\{v_m\}_m$) satisfies
\[
v_m\longrightarrow v\quad\textrm{ locally uniformly in }\mathbb{R}^n. 
\]

Let us then verify that such $v$ satisfies the hypothesis of Theorem \ref{thm.liou}. First, passing to the limit in \eqref{eq.growthm} we see that 
\[
[v]_{C^\alpha(B_R)}\leq CR^\gamma.
\]
so the growth condition is satisfied. 

Next, recall that $u_m$ satisfies
\begin{equation}\label{eq.imu}
I_mu_m=f_m\quad\textrm{ in }\quad B_1
\end{equation}
in the viscosity sense for any $m\geq 1$. Notice also that the operators of order $2s$ that make up the class $\mathcal{L}$ vanish identically when evaluated in linear functions (for $s>1/2$) or constants (for any $s$) so we have, using the scaling, 
\[
I_mv_m=\tilde{f}_m:=\frac{f_m}{\theta(r'_m)}\quad\textrm{ in }\quad B_{1/r'_m}.
\]
Then, applying Lemmas \ref{lem.stab2} and \ref{lem.stab3} we get 
\[
Iv=0\quad\textrm{ in }\quad \R^n
\]
and therefore we may assume it is $C^{2,\varepsilon}$ for some $\varepsilon>0$, for instance, applying the approximation argument Lemma 2.1 in \cite{CS3}. 

On the other hand, according to the definition of ellipticity \eqref{eq.ellip} (and recalling that $\|f_m\|_{L^\infty(B_1)}\leq 1$) we get that for $x\in B_{1/2}(z_{k_m})$ and $|h|<1/2$ \eqref{eq.imu} implies
\[
\mathcal{M}^+(u_m(x+h)-u_m(x))\geq -1 \quad \textrm{ and }\quad \mathcal{M}^-(u_m(x+h)-u_m(x))\leq 1. 
\]

Also, using the scaling of the equations (see comment after Definition \ref{def.visc}),
\[
\mathcal{M}^+(v_m(x+h)-v_m(x))\geq \frac{-(r'_m)^{2s}}{(r'_m)^{\alpha+\gamma}\theta(r'_m)}
\]
and
\[
\mathcal{M}^-(v_m(x+h)-v_m(x))\leq \frac{(r'_m)^{2s}}{(r'_m)^{\alpha+\gamma}\theta(r'_m)}
\]
for $|x|\leq \frac{1}{2r'_m}$ in the viscosity sense. 

Now, it is clear that 
\[
v_m(\cdot+h)-v_m(\cdot)\longrightarrow v(\cdot+h)-v(\cdot)\quad\textrm{ locally uniformly in }\mathbb{R}^n
\]
and by our definition of $\theta$, $\alpha$, $\gamma$ and $r'_m$
\[
\frac{(r'_m)^{2s}}{(r'_m)^{\alpha+\gamma}\theta(r'_m)}\longrightarrow 0 
\]
as $m\rightarrow\infty$. 

Also, by the Dominated Convergence Theorem, the pointwise convergence of $v_m(\cdot+h)-v_m(\cdot)$ to $v(\cdot+h)-v(\cdot)$ and the growth condition satisfied by both $v_m$ and $v$ (notice that $\gamma<2s$), it follows that
\[
\int_{\mathbb{R}^n}\frac{|(v_m(x+h)-v_m(x))-(v(x+h)-v(x))|}{1+|x|^{n+2s}}\:dx\longrightarrow 0
\]
as $m\rightarrow\infty$, so taking
\[
w_m(\cdot):=v_m(\cdot+h)-v_m(\cdot)\quad\textrm{ and }\quad w(\cdot):=v(\cdot+h)-v(\cdot)
\]
in Lemma \ref{lem.stab} we obtain that the increments of $v$ are global solutions (applying such Lemma to larger and larger values of $R$), i.e. $v$ is a viscosity (and since it is $C^2$ pointwise) solution of
\[
\mathcal{M}^+ (v(\cdot+h)-v(\cdot)) \geq 0\quad\textrm{ and }\quad \mathcal{M}^- (v(\cdot+h)-v(\cdot))\leq 0 \quad\textrm{ in }\quad \mathbb{R}^n. 
\]

Finally, we need to verify condition (2) of Theorem \ref{thm.liou}. Let $\mu\in L^1(\mathbb{R}^n)$ with compact support and $\int_{\mathbb{R}^n}\mu(h)\:dh=1$. Then using Jensen's inequality we get
\begin{align*}
-2\|f_m\|_{L^\infty(B_1)}&\leq  \int Iu_m(\cdot+h)\mu(h)\:dh-Iu_m(\cdot) \\
   & \leq I\left(\int u_m(\cdot+h)\mu(h)\:dh\right) -Iu_m(\cdot) \\
   & \leq \mathcal{M}^+\left(\int u_m(\cdot+h)\mu(h)\:dh -u_m(\cdot)\right) \\
   & = \mathcal{M}^+\left(\int u_m(\cdot+h) -u_m(\cdot)\mu(h)\:dh\right) \\
\end{align*}
and proceeding just as in the previous step we get
\[
0 \leq \mathcal{M}^+_{\mathcal{L}}\Big(\int v(\cdot+h)-v(\cdot)\mu(h)\:dh\Big)\quad\textrm{ in }\quad \mathbb{R}^n
\]
as we wanted.
 
Therefore, we can apply Theorem \ref{thm.liou} to get that 
\[
v(x)=p\cdot x+q
\]
for some $p\in\mathbb{R}^n$ and $q\in\mathbb{R}$. Finally, passing to the limit in \eqref{eq.ortm} we obtain that $v\equiv0$. But passing to the limit in \eqref{eq.nondeg} we get that $[v]_{C^\alpha(B_1)}\geq 1/2$, thus arriving to a contradiction.    
\end{proof}

\begin{proof}[Proof of Proposition \ref{prop.blow2}]
The proof follows as the previous one considering $\gamma:=2s-\alpha-\varepsilon$ and following the same steps. 
\end{proof}

\begin{rem}\label{rem.c}
It is easy to check that the proofs above work \emph{mutatis mutandis} for operators with lower order terms of the form 
\begin{equation}\label{eq.Lc}
Iu(x)=\inf_{\alpha\in\mathcal{A}}\left\lbrace L_\alpha u(x)+c_\alpha(x)\right\rbrace
\end{equation}
as long as we have a uniform bound on the $L^\infty$ norm of $c_\alpha$, i.e. 
\[
\|c_\alpha\|_{L^\infty(B_1)}\leq M\quad\textrm{ for all }\alpha\in\mathcal{A}
\]
for some $M>0$. Should that be the case, the constant $M$ also appears on the right hand side of \eqref{eq.c2s} (or \eqref{eq.c2sep}).
\end{rem}

With the aid of the previous propositions and Remark \ref{rem.c} we are in position to give the proof of our main result:

\begin{proof}[Proof of Theorem \ref{thm.c2s}]

We will proof \eqref{eq.2} using Proposition \ref{prop.blow1}; the proof of \eqref{eq.3} is completely analogous using Proposition \ref{prop.blow2} instead.

First, let us consider a cutoff function $\eta\in C^\infty_c(B_1)$ (say $\eta\equiv0$ outside $B_{7/8}$) such that $0\leq \eta\leq 1$ and $\eta \equiv 1$ in $B_{3/4}$. Notice that by the estimates in \cite{Se1} we have an interior $C^\alpha$ estimate for $u$ (in terms of $\|u\|_{L^\infty(\R^n)}$ and $\|f\|_{L^\infty(B_1)}$) and hence the function $\eta u$ belongs to $C^\alpha(\R^n)$ for some $\lfloor 2s\rfloor<\alpha<2s$. Moreover, since $1-\eta$ vanishes in $B_{3/4}$, by letting $\tilde{u}:=(1-\eta)u$ we see that we can control
\begin{align*}
|L_\alpha \tilde{u}(x)| & = \Big|\int_{\mathbb{R}^n}(\tilde{u}(x+y)+\tilde{u}(x-y)-2\tilde{u}(x))K_\alpha(y)\:dy\Big| \\ 
						& \leq \Lambda\int_{\mathbb{R}^n\setminus B_{1/8}}\frac{|\tilde{u}(x+y)+\tilde{u}(x-y)-2\tilde{u}(x)|}{|y|^{n+2s}}\:dy \\ 
						& \leq C\|\tilde{u}\|_{L^\infty(\mathbb{R}^n\setminus B_{1/8})}\leq C\|u\|_{L^\infty(\mathbb{R}^n)}
\end{align*}
for any $x\in B_{5/8}$.

Therefore, we see that 
\[
f(x)=Iu(x)=I(\eta u+(1-\eta)u)(x)=\inf_{\alpha\in\mathcal{A}}\left\lbrace L_\alpha (\eta u)(x)+c_\alpha(x)\right\rbrace\quad\textrm{ in }\quad B_{5/8}
\]
in the viscosity sense with 
\[
c_\alpha(x):=L_\alpha \tilde{u}(x)
\]
so that $\eta u$ satisfies an equation with an operator of the form \eqref{eq.Lc} with uniformly bounded lower order terms and hence by rescaling Proposition \ref{prop.blow1} we get
\[
[\eta u]_{C^{2s}(B_{5/16})}\leq C\left(\|f\|_{L^\infty(B_{5/8})}+[\eta u]_{C^\alpha(\mathbb{R}^n)}+\|u\|_{L^\infty(\mathbb{R}^n)}\right).
\] 
Next, we note that 
\[
[u]_{C^{2s}(B_{5/16})}=[\eta u]_{C^{2s}(B_{5/16})}\quad\textrm{ and }\quad [\eta u]_{C^\alpha(\mathbb{R}^n)}=[\eta u]_{C^\alpha(B_1)}.
\]
so that the previous estimate can be rewritten as follows
\begin{equation}\label{eq.estimate}
[u]_{C^{2s}(B_{5/16})}\leq C\left(\|f\|_{L^\infty(B_{5/8})}+\|\eta u\|_{C^\alpha(B_1)}+\|u\|_{L^\infty(\mathbb{R}^n)}\right).
\end{equation}

We ought to control the second term on the right hand side of \eqref{eq.estimate}. To this end, we will use the \emph{interior norms (and seminorms)} which are defined as follows (see Gilbarg and Trudinger \cite{GT}): if $\gamma=k+\gamma'$ (with $k=\lfloor\gamma\rfloor$) 
\[
[w]_{\gamma,\Omega}^{(\delta)}:=\sup_{\substack{x\neq y \\ x,y\in\Omega}}\left(d_{x,y}^{\gamma+\delta}\frac{|D^kw(x)-D^kw(y)|}{|x-y|^{\gamma'}}\right).
\]
and 
\[
\|w\|_{\gamma,\Omega}^{(\delta)}:=\sum_{j=1}^k\sup_{x\in\Omega}\left(d_x^{j+\delta}|D^jw(x)|\right)+[w]_{\gamma,\Omega}^{(\delta)}
\]
where
\[ 
d_x:=\textrm{dist}(x,\partial\Omega)\quad\textrm{ and }\quad d_{x,y}:=\min\{d_x,d_y\}.
\]

Using these norms, we can rescale \eqref{eq.estimate} and apply it to balls of radius $\tau>0$. Taking the supremum over all such balls satisfying $B_{2\tau}\subset B_1$ we get
\[
[u]_{2s,B_1}^{(0)}\leq C\left(\|f\|_{0,B_1}^{(2s)}+\|\eta u\|_{\alpha,B_1}^{(0)}+\|u\|_{L^\infty(\mathbb{R}^n)}\right)
\]
and from here deduce
\[
\|u\|_{2s,B_1}^{(0)}\leq C\left(\|f\|_{0,B_1}^{(2s)}+\|u\|_{L^\infty(\mathbb{R}^n)}\right).
\]
In particular
\[
\|u\|_{C^{2s}(B_{1/2})}\leq C\left(\|f\|_{L^\infty(B_1)}+\|u\|_{L^\infty(\mathbb{R}^n)}\right)
\]
which is what we wanted to prove. 
\end{proof}

\section{Application: nonlocal two membranes problem}\label{sec.two}

In this section we give an application of Theorem \ref{thm.c2s} to the two membranes problem. The two membranes problem has attracted much interest in the PDE community in the past years, see for instance \cite{CDS} and \cite{CDV} and the references therein. In \cite{CDV} a ``bid and ask'' model from Mathematical Finance is considered where an asset has a price that varies randomly and a buyer and a seller have to agree on a price for a transaction to take place. A local (second order) approach to the problem is discussed there. However, it is often the case that the price of an asset can have jumps, in the sense that it can change abruptly, see \cite{CT16}. Hence, it makes sense to study the problem from \cite{CDV} but taking into account a ``long range interactions'' model, which of course is done through nonlocal equations. The problem has the following form:

\begin{equation}\label{eq.nmain}
  \left\{ \begin{array}{rcll}
  u & \geq & v & \textrm{ in } B_1 \\
  I_1u & \leq & f & \textrm{ in } B_1 \\
  I_2v & \geq & g & \textrm{ in } B_1 \\
  I_1u & = & f & \textrm{ in } B_1\cap\Omega \\
  I_2v & = & g & \textrm{ in } B_1\cap\Omega \\
  u & = & u_0 & \textrm{ in } \mathbb{R}^n\setminus B_1 \\ 
  v & = & v_0 & \textrm{ on } \mathbb{R}^n\setminus B_1. 
\end{array}\right.
\end{equation}

Here $\Omega:=\{u>v\}$ and $I_1$, $I_2$ are two fully nonlinear operators with $I_1$ convex and $I_2v=-I_1(-v)$ (hence concave) of the form
\begin{equation}\label{eq.ops}	
I_1u(x):=\sup_{\alpha\in\mathcal{A}} L_\alpha u(x)\quad\textrm{ and }\quad I_2v(x):=\inf_{\alpha\in\mathcal{A}} L_\alpha v(x)
\end{equation}
i.e. of the form \eqref{eq.L}-\eqref{eq.ker}.

It is woth pointing out that the nonlocal theory for obstacle-type problems such as \eqref{eq.nmain} is less developed than its local counterpart, so that new challenges and difficulties arise. Indeed, a natural first step when dealing with problems like \eqref{eq.nmain} is to prove a ``basic'' (in the sense that is not expected to be optimal) regularity estimate. In the local case this is achieved by showing that the operator is bounded and hence solutions belong to the space $W^{2,p}$ for any $p<\infty$, which in turn implies that solutions are in $C^{1,\alpha}$ for any $\alpha<1$ by the Sobolev Embedding Theorem. Such a reasoning is not trivially adapted to the nonlocal setting; indeed, there is a lack of estimates analogous to the local $W^{2,p}$ ones and there are only partial results, mainly the one proved by Yu in \cite{Y}. 

Here, however, we are able to circumvent this issue in our case and give a basic regularity estimate that follows as an easy consequence of Theorem \ref{thm.c2s}. The idea is to ``skip the $W^{2,p}$ estimate'' and get directly the type of regularity that would follow form the Sobolev Embedding.

\begin{thm}
Let $u$ and $v$ be bounded solutions of \eqref{eq.nmain} in the viscosity sense and $f,g\in C(\overline B_1)$. Then 
\begin{enumerate}
\item if $s\neq 1/2$ we have $u,v\in C^{2s}(B_{1/2})$ and 
\begin{equation}\label{eq.bas}
\|u\|_{C^{2s}(B_{1/2})},\|v\|_{C^{2s}(B_{1/2})}\leq C
\end{equation} 

\item if $s= 1/2$, for any $\varepsilon>0$ we have $u,v\in C^{2s-\varepsilon}(B_{1/2})$ and 
\begin{equation}\label{eq.basep}
\|u\|_{C^{2s-\varepsilon}(B_{1/2})},\|v\|_{C^{2s-\varepsilon}(B_{1/2})}\leq C.
\end{equation} 
\end{enumerate}
The constant in \eqref{eq.bas} depends only on $n$, $s$, $\lambda$, $\Lambda$, $\|u\|_{L^\infty(\mathbb{R}^n)}$, $\|v\|_{L^\infty(\mathbb{R}^n)}$, $\|f\|_{L^\infty(B_1)}$, $\|g\|_{L^\infty(B_1)}$ and the constant in \eqref{eq.basep} depends also on $\varepsilon$. 
\end{thm}

\begin{proof}
We prove the result for $u$, the proof for $v$ is analogous. We will show that $|I_1u|\leq C$ in the viscosity sense for some universal constant $C$. The result will then follow from Theorem \ref{thm.c2s}. 

Let $\varphi$ function touching $u$ by below at $x_0\in B_1$ and $C^2$ in some small neighborhood of $x_0$, coinciding with $u$ outside. Recall that $u$ is a viscosity supersolution across the whole ball (disregarding if $x_0$ is in the contact set $\{u=v\}$ or not), so we have  

\[
I_1\varphi(x_0)\leq f(x_0)\leq\|f\|_{L^\infty(B_1)}. 
\]

If instead $\varphi$ touches $u$ by above, we separate two cases:

\textbf{Case 1:} if $x_0\in\Omega$ then $u$ is also a subsolution and we get 

\[
I_1\varphi(x_0)\geq f(x_0)\geq-\|f\|_{L^\infty(B_1)}
\]
as before. 

\textbf{Case 2:} if $x_0\notin\Omega$, notice that $\varphi$ also touches $v$ by above and lies above $v$ (because it lies above $u$). Hence $\varphi$ would lie above the corresponding test function that coincides with $v$ outside a neighborhood (and coincide with it at $x_0$). Also $v$ is a subsolution for $I_2$ across the whole ball. Then 

\[
I_2\varphi(x_0)\geq g(x_0)\geq-\|g\|_{L^\infty(B_1)}.
\]

But then
\[
I_1 \varphi(x_0)\geq g(x_0)\geq-\|g\|_{L^\infty(B_1)}
\]
and we are done.
\end{proof}

\begin{rem}
From the proof it follows that it would actually suffice $I_1\geq I_2$ to achieve the $C^{2s}$ regularity. 
\end{rem}

{\bf Acknowledgments}
The author was partially supported by a Conicet scholarship. 

The author would like to thank professor Caffarelli for all his guidance and support during his PhD years.


\begin{thebibliography}{999}

\bibitem {Ap} D. Applebaum, L\'evy processes --- from probability to finance and quantum groups, Notices AMS {\bf 51} (2004), no. 11, 1336-1347.

\bibitem {BV16} C. Bucur and E. Valdinoci, \textit{Nonlocal Diffusions and Applications}, Lecture Notes of the Unione Matematica Italiana, Springer, 2016.

\bibitem {CC} L. Caffarelli and X. Cabr\'e, \emph{Fully Nonlinear Elliptic Equations}, AMS Colloquium Publications, Vol 43, 1995.

\bibitem {CDS} L. Caffarelli, D. De Silva and O. Savin, The two membranes problem for different operators, Annales de l'Institut Henri Poincare (C) Non Linear Analysis {\bf 34} (2017), no. 4, 899-932.

\bibitem {CDV}  L.A. Caffarelli, L. Duque,H. and Vivas, The two membranes problem for fully nonlinear operators, Discrete $\&$ Continuous Dynamical Systems - A 2018 {\bf 38}, no. 12, 6015-6027. 

\bibitem {CS1} L. Caffarelli and L. Silvestre, Reglarity theory for fully nonlinear integro-differential equations, Comm. Pure Appl. Math. {\bf 62} (2009), no. 5, 597-638.

\bibitem {CS2} L. Caffarelli and L. Silvestre, Regularity results for nonlocal equations by approximation, Arch. Rational Mech. Anal. {\bf 200} (2011), no. 1, 59-88.

\bibitem {CS3} L. Caffarelli and L. Silvestre, The Evans-Krylov theorem for nonlocal fully nonlinear equations, Annals of Mathematics {\bf 174} (2011), no.2, 1163-1187.

\bibitem {CD} H. Chang Lara and G. D\'avila, Regularity for solutions of nonlocal, nonsymmetric equations,  Annales de l'Institut Henri Poincare (C) Non Linear Analysis {\bf 29} (2012), no. 6, 833-859.

\bibitem {CT16} R. Cont and P. Tankov, \emph{Financial Modelling With Jump Processes}, Financial Mathematics Series, Chapman \& Hall/CRC 2004.

\bibitem {FR1} X. Fern\'andez-Real and X. Ros-Oton, Regularity theory for general stable operators: parabolic equations, J. Funct. Anal {\bf 272} (2017), no. 10, 4165-4221.

\bibitem {Ga} N. Garofalo, Fractional thoughts, preprint (2017), https://arxiv.org/abs/1712.03347.

\bibitem {GT} D. Gilbarg, N. S. Trudinger, \emph{Elliptic Partial Differential Equations of Second Order}, Springer, 2015.

\bibitem {humphries} N. E. Humphries et al., Environmental context explains Levy and Brownian movement patterns of marine predators, Nature {\bf 465} (2010), no. 7301, 1066-1069.

\bibitem {K} D. Kriventsov, $C^{1,\alpha}$ regularity for nonlinear nonlocal elliptic equations with rough kernels, Comm. Partial Differential Equations {\bf 12} (2013), no. 12, 2081-2106.

\bibitem {Land} N. S. Landkof, \emph{Foundations of Modern Potential Theory}, Springer 1972.

\bibitem {Ro1} X. Ros-Oton, Nonlocal elliptic equations in bounded domains: a survey, Publ. Mat. {\bf 60} (2016), no. 1, 3-26.

\bibitem {RS1} X. Ros-Oton and J. Serra, Regularity theory for general stable operators, J. Differential Equations {\bf 260} (2016), no. 12, 8675-8715.

\bibitem {RS2} X. Ros-Oton and J. Serra, Boundary regularity for fully nonlinear integro-differential equations, Duke Math. J. {\bf 165} (2016), no. 11, 2079-2154.

\bibitem {RV} X. Ros-Oton and H. Vivas, Higher-order boundary regularity estimates for nonlocal parabolic equations, Calc. Var. Partial Differential Equations, {\bf 57} (2018), no.5 111-131.

\bibitem {Se1} J. Serra, Regularity for fully nonlinear nonlocal parabolic equations with rough kernels, Calc. Var. Partial Differential Equations {\bf 54} (2015), no. 1, 615-629.

\bibitem {Se2} J. Serra, $C^{\sigma+\alpha}$ regularity for concave nonlocal fully nonlinear elliptic equations with rough kernels, Calc. Var. Partial Differential Equations {\bf 54} (2015), no. 4, 3571-3601.

\bibitem {St} E. Stein, \textit{Singular integrals and differentiability properties of functions}, Princeton Mathematical Series No. 30, Princeton University Press 1970.

\bibitem {Y} H. Yu, $W^{\sigma,\epsilon}$-estimates for nonlocal elliptic equations, Annales de l'Institut Henri Poincare (C) Non Linear Analysis {\bf 34} (2017), no. 5, 1141-1153.

\end{thebibliography}
\end{document}